\documentclass[12pt,reqno]{amsart}

\usepackage{amssymb}
\usepackage{amsmath}
\usepackage{latexsym}
\usepackage{amsthm}
\usepackage{epsfig}
\usepackage{color}
\usepackage{comment}
\usepackage{amsfonts,amssymb,mathrsfs,amscd}

\usepackage{enumitem}

\theoremstyle{plain}
\newtheorem{theorem}{Theorem}[section]
\newtheorem{proposition}{Proposition}[section]
\newtheorem{corollary}{Corollary}[section]

\newtheorem{remark}{\bf Remark}[section]
\theoremstyle{definition}

\newcommand{\rot}{\mathop\mathrm{rot}}
\newcommand{\grad}{\mathop\mathrm{grad}}
\renewcommand{\div}{\mathop\mathrm{div}}
\newcommand{\Tr}{\mathop\mathrm{Tr}}

\def\({\left(}
\def\){\right)}

\def\rot{\operatorname{rot}}

\def\Tr{\operatorname{Tr}}

\begin{document}

 \title[Interpolation inequalities on the sphere]
 {On a class of interpolation inequalities on the 2D sphere}

\author[ A. Ilyin, 
and  S. Zelik] {
Alexei Ilyin${}^1$,
and Sergey
Zelik${}^{2,3}$}


\email{ilyin@keldysh.ru}
\email{s.zelik@surrey.ac.uk}
\address{${}^1$ Keldysh Institute of Applied Mathematics, Moscow, Russia}
\address{${}^2$ University of Surrey, Department of Mathematics, Guildford, GU2 7XH, United Kingdom.}
\address{${}^3$ \phantom{e}School of Mathematics and Statistics, Lanzhou University, Lanzhou\\ 730000,
P.R. China}

\begin{abstract}
We prove   estimates for the $L^p$-norms of systems of
functions and divergence free vector functions that are orthonormal in the Sobolev space  $H^1$ on the 2D sphere.
As a corollary, order sharp constants in the embedding $H^1\hookrightarrow L^q$, $q<\infty$,
are obtained in the  Gagliardo--Nirenberg interpolation inequalities.
\end{abstract}

\maketitle

\section{Introduction}
The following interpolation inequality
holds on the sphere $\mathbb{S}^d$ (see~\cite{Beckner} and also \cite{Veron}):
\begin{equation}\label{1}
\frac{q-2}d\int_{\mathbb{S}^d}|\nabla\varphi|^2d\mu+
\int_{\mathbb{S}^d}|\varphi|^2d\mu\ge
\left(\int_{\mathbb{S}^d}|\varphi|^qd\mu\right)^{2/q}.
\end{equation}
Here $d\mu$ is the normalized  Lebesgue measure
on $\mathbb{S}^d$:
$$
d\mu=\frac{d\sigma}{\sigma_d}=\frac{d\sigma}{\frac{2\pi^{\frac{d+1}2}}{\Gamma(\frac{d+1}2)}}\,,
$$
so that $\mu(\mathbb{S}^d)=1$ (the gradient is calculated with respect to the natural
metric). Next, $q\in[2, \infty)$ for $d=1,2$,
and $q\in[2, 2d/(d-2)]$ for $d\ge 3$.
The remarkable fact about~\eqref{1} is that the constant $(q-2)/d$ is sharp
for all admissible  $q$.
The inequality clearly degenerates and  turns into equality on constants. The fact that
the constant $(q-2)/d$ is sharp is verified by means of the sequence
$\varphi_\varepsilon(s)=1+\varepsilon v(s)$ as $\varepsilon\to0$,
where $v(s)$ is an eigenfunction of the Laplacian on $\mathbb{S}^d$
corresponding to the first positive eigenvalue $d$, see \cite{DEKL} and the references therein.

However, in applications (for instance, for the Navier--Stokes equations
on the 2D sphere) the functions $\varphi$ usually play the role
of stream functions of a divergence free vector functions $u$, $u=\nabla^\perp\varphi$,
and  therefore
without loss of generality $\varphi$
can be chosen to be orthogonal to constants.

In this work we consider the two-dimensional sphere $\mathbb{S}^2$
only and are interested in writing the Sobolev embedding
$H^1(\mathbb{S}^2)\hookrightarrow L^q(\mathbb{S}^2)$ as a
multiplicative inequality of Gagliardo--Nirenberg type
involving the $L^2$-norms of $\varphi$ and $\nabla\varphi$ on right-hand side:
$\|\varphi\|_{L^2(\mathbb{S}^2)}=:\|\varphi\|$ and
$\|\nabla \varphi\|_{L^2(\mathbb{S}^2)}=:\|\nabla\varphi\|$.

It is also well known that in the case of $\mathbb{R}^d$
 interpolation inequalities in the additive form and in the multiplicative form are
 equivalent and the passage from the former to the latter
 is realized by the introduction of the parameter $m$ in the inequality
 (by scaling $x\to mx$) and
subsequent minimization with respect to $m$.  To go other way round
one can use Young's inequality (with parameter) for products
to obtain the interpolation inequality in the additive form.

This scheme obviously  does not work on a manifold
 due to the lack of scaling.
One possible way to introduce a parameter in the Sobolev inequality
is to consider the Sobolev space $H^1$ with norm and scalar product
$$
\|\varphi\|_{H^1}^2:=m^2\|\varphi\|^2+\|\nabla\varphi\|^2,\quad
(\varphi_1,\varphi_2)_{H^1}:=m^2(\varphi_1,\varphi_2)+
(\nabla\varphi_1,\nabla\varphi_2)
$$
depending on a parameter $m>0$, and then to trace down the explicit dependence of
the embedding constant on $m$. In this work this is done in much
more general framework of the inequalities for $H^1$-orthonormal
families proved in \cite{LiebJFA}.

We can now state and discuss our main result.

\begin{theorem}\label{Th:intro}
Let a family of zero mean functions
$\{\varphi_j\}_{j=1}^n\in\dot{H}^1(\mathbb{S}^2)$ be orthonormal
with respect to the scalar product
\begin{equation}\label{orth-m}
m^2(\varphi_i,\varphi_j)+(\nabla\varphi_i,\nabla\varphi_j)=\delta_{ij}.
\end{equation}
Then for $1\le p<\infty$ the function
$$
\rho(x):=\sum_{j=1}^n|\varphi_j(x)|^2
$$
satisfies the inequality
\begin{equation}\label{Liebd2}
\|\rho\|_{L^p}\le\mathrm{B}_pm^{-2/p}n^{1/p},
\end{equation}
where
\begin{equation}\label{Liebbp}
\mathrm{B}_p\le\left(\frac{p-1}{4\pi}\right)^{(p-1)/p}.
\end{equation}
\end{theorem}

These inequalities were proved in the case of $\mathbb{R}^d$
in~\cite{LiebJFA} for $p=\infty$ ($d=1$),  $1\le p<\infty$ ($d=2$),
and for the critical $p=d/(d-2)$ ($d\ge3$). No expressions for the
constants were given, the dependence on $m$ is again uniquely defined by
scaling, and the main interest there was in the dependence of the
right hand side on $n$.

For $p=2$ this  inequality has played an essential role in finding
explicit optimal bounds for the attractor dimension for the damped
regularized Euler--Bardina--Voight system for various boundary
conditions both in the two and three dimensional cases, see
\cite{IZLap70,IKZPhysD,MZ}. More precisely, it was shown in \cite{IZLap70,MZ} that
$\mathrm{B}_2\le(4\pi)^{-1/2}$  for $\mathbb{T}^2$,
$\mathbb{S}^2$, and $\mathbb{R}^2$ based on  the following two inequalities for the lattice
sum over $\mathbb{Z}^2_0=\mathbb{Z}^2\setminus\{0,0\}$ and the
series with respect to the spectrum of the Laplacian on
$\mathbb{S}^2$ that were proved there for the special case, when  $p=2$
\begin{eqnarray}
  J_p(m):= \frac{(p-1)m^{2(p-1)}}\pi\sum_{n\in{\mathbb Z}_0^2}\frac1{(m^2+|n|^2)^p}<1, \label{T2} \\
   I_p(m):=m^{2(p-1)}(p-1)\sum_{n=1}^\infty\frac{2n+1}{\bigl(m^2+n(n+1)\bigr)^p}<1. \label{S2}
\end{eqnarray}

The case $p=2$ is not at all specific in the general scheme of the
proof of Theorem~\ref{Th:intro} and the general case
in the theorem both for $\mathbb{T}^2$ and $\mathbb{S}^2$
 immediately follows once we have inequality~\eqref{T2}, \eqref{S2}
 for all $1< p<\infty$.

 Inequality \eqref{T2} and therefore Theorem~\ref{Th:intro}
 for the torus $\mathbb{T}^2$ has recently been proved
 in~\cite{Lieb90arxiv}, and  the main result  of this work is
 the proof of \eqref{S2} and Theorem~\ref{Th:intro} for the sphere.

We point out that in the case of $\mathbb{R}^2$, instead of
\eqref{T2} and \eqref{S2} we simply have the equality
\begin{equation}\label{R2eq}
\frac{(p-1)m^{2(p-1)}}\pi\int_{\mathbb {R}^2}\frac{dx}{(m^2+|x|^2)^p}=1.
\end{equation}

For one function ($n=1$) Theorem~\ref{Th:intro} is equivalent to
the Sobolev inequality with parameter $H^1\hookrightarrow L^q$,
$q=2p\in [2,\infty)$,  which can equivalently be
written as a Gagliardo--Nirenberg inequality
\begin{equation}\label{2}
\|f\|_{L^q}\le\left(\frac{1}{4\pi}\right)^{(q-2)/2q}
 \left(\frac q2\right)^{1/2}\|f\|^{2/q}\|\nabla f\|^{1-2/q},
\end{equation}
which holds for $\mathbb{R}^2$, $\mathbb{T}^2$  and $\mathbb{S}^2$,
see Corollary~\ref{Cor:GN}.

For the torus $\mathbb{T}^2$ inequality \eqref{2} can be proved
in a direct way \cite{Lieb90arxiv} by using the Hausdorff--Young inequality
for the discrete Fourier series and again estimate~\eqref{T2}. In the case of $\mathbb{R}^2$
this approach is well known and with the additional use of the Babenko--Beckner
inequality \cite{Babenko, Beckner_Fourier} for the Fourier transform (and equality~\eqref{R2eq})
gives the following improvement of inequality \eqref{2}
for $\mathbb{R}^2$ with the best to date closed form  estimate for the constant \cite{Nasibov}:
\begin{equation}\label{Gag-Nir-R2}
\|\varphi\|_{L^q(\mathbb{R}^2)}\le\left(\frac1{4\pi}\right)^\frac{q-2}{2q}
\frac{q^{(q-2)/q}}{(q-1)^{(q-1)/q}}
\left(\frac q2\right)^{1/2}
\|\varphi\|^{2/q}\|\nabla\varphi\|^{1-2/q},\ q\ge2,
\end{equation}
see also \cite[Theorem 8.5]{Lieb--Loss} where the  equivalent result is obtained
for the inequality in the additive form.

Of course, inequality~\eqref{Gag-Nir-R2} for $\mathbb{R}^2$ and inequality~\eqref{Gag-Nir}
for $\mathbb{T}^2$ both  are a special case of Gagliardo--Nirenberg inequality.
For $\mathbb{R}^2$ the best constant is known for every $q\ge2$ and is expressed
in terms of a norm of the ground state solution
of the corresponding nonlinear Euler--Lagrange equation~\cite{Wein83}.
However, not in the explicit form. As mentioned above,
inequality~\eqref{Gag-Nir-R2} was  known before,
while  inequality \eqref{Gag-Nir}
(more precisely, the estimate for the constant in it)
for the torus  $\mathbb{T}^2$ was recently obtained in~\cite{Lieb90arxiv}.

As far as the case of the sphere $\mathbb{S}^2$ is concerned we do not know
how to prove \eqref{2} in a  way  other than the one function corollary
of the general Theorem~\ref{Th:intro}. The main difference from the case
of $\mathbb{T}^2$ is that the orthonormal spherical functions are
not uniformly bounded in $L^\infty$.

Our approach  makes it possible to
prove similar inequalities in  the vector case. Namely, we show that for
 $u\in \mathbf{H}^1_0(\Omega)\cap\{\div u=0\}$
it holds
\begin{equation*}\label{2-vec}
\|u\|_{L_q(\mathbb{S}^2)}\le\left(\frac{1}{4\pi}\right)^{(q-2)/2q}
 \left(\frac q2\right)^{1/2}\|u\|^{2/q}\|\rot u\|^{1-2/q}.
\end{equation*}
Here $\Omega\subseteq\mathbb{S}^2$ is an arbitrary domain on
$\mathbb{S}^2$. This inequality looks very similar to \eqref{2},
the important difference being that, unlike the scalar case, the
vector Laplacian on $\mathbb{S}^2$ is positive definite, and we can freely use
the extension by zero.

 Finally, it is natural to compare inequalities \eqref{1} with $d=2$
 and \eqref{2} for functions with mean value zero. To do so we go over
to the natural measure
 on $\mathbb{S}^2$ in \eqref{1} and then use the Poincare inequality
$\|\varphi\|^2\le2^{-1}\|\nabla\varphi\|^2$ to obtain:
\begin{equation*}\label{3}
\aligned
\|\varphi\|_{L^q(\mathbb{S}^2)}\le\left(\frac{1}{4\pi}\right)^{(q-2)/2q}
\left(\frac{q-2}2\|\nabla\varphi\|^2+\|\varphi\|^2\right)^{1/2}\le\\\le
\left(\frac{1}{4\pi}\right)^{(q-2)/2q}
\left(\frac{q-1}2\right)^{1/2}\|\nabla\varphi\|,
\endaligned
\end{equation*}
while \eqref{2} gives
\begin{equation*}\label{4}
\|\varphi\|_{L_q(\mathbb{S}^2)}\le\left(\frac{1}{4\pi}\right)^{(q-2)/2q}
 \left(\frac q2\right)^{1/2}\frac1{2^{1/q}}\|\nabla \varphi\|.
\end{equation*}
The constant here is marginally smaller, since
$$
2^{-2/q}\le1-1/q, \quad q\ge2.
$$

Since  inequality~\eqref{1} turns into equality on  constants, this inequality
may not be sharp on the subspace of  zero mean functions on $\mathbb{S}^2$,
and the constant in~\eqref{2} is not sharp. However, looking at
\eqref{2} and \eqref{Gag-Nir-R2} for $\mathbb{T}^2$,  $\mathbb{S}^2$
and for $\mathbb{R}^2$, respectively, one can suggest that
that the \emph{sharp} constant here is
\begin{equation*}\label{cq}
\mathrm{c}_q\sim\left(\frac1{8\pi}\right)^{1/2}\,q^{1/2}\quad\text{as}\quad q\to\infty.
\end{equation*}

The expression  on the right-hand side here  curiously coincides
with  sharp constant in the Sobolev inequality for the limiting
exponent, see \cite{Talenti,LiebAnnals}:
$$
\|\varphi\|_{L^q({\mathbb{R}^d})}\le
\frac{\sqrt{q}}{d\sqrt{2\pi}}\left[\frac{\Gamma(d)}{\Gamma(d/2)}\right]^{1/d}
\|\nabla\varphi\|_{L^2({\mathbb{R}^d})}, \qquad\frac1q=\frac12-\frac1d,
$$
if we formally set  $d=2$.
Of course, this inequality does hold in $\mathbb{R}^2$, since  $d\ge3$ in it.

Theorem~\ref{Th:intro} and the similar result in the vector case are proved in the
next Section~\ref{Sec:2},
and the key estimate for the series \eqref{S2} is proved in Section~\ref{Sec:3}.

 \setcounter{equation}{0}
\section{Proof of the main result}\label{Sec:2}

\begin{proof}[Proof of Theorem~\ref{Th:intro}]
We  first
 recall the basic facts concerning the spectrum of the
scalar Laplace operator $\Delta=\div\nabla$ on the sphere
$\mathbb{S}^{2}$ (see, for instance, \cite{S-W}):
\begin{equation}\label{harmonics}
-\Delta Y_n^k=n(n+1) Y_n^k,\quad
k=1,\dots,2n+1,\quad n=0,1,2,\dots.
\end{equation}
Here the $Y_n^k$ are the orthonormal real-valued spherical
harmonics and each eigenvalue $\Lambda_n:=n(n+1)$ has multiplicity $2n+1$.

The following identity is essential in what
follows: for any $s\in\mathbb{S}^{2}$
\begin{equation}\label{identity}
\sum_{k=1}^{2n+1}Y_n^k(s)^2=\frac{2n+1}{4\pi}.
\end{equation}
Since inequality \eqref{Liebd2} with \eqref{Liebbp} clearly holds for $p=1$
we assume below that $1<p<\infty$. Let us define two operators
\begin{equation}\label{HH}
 \mathbb{H}= V^{1/2}(m^2-{\Delta})^{-1/2}\Pi,\quad
  \mathbb{H}^*=\Pi(m^2-{\Delta})^{-1/2}V^{1/2},
\end{equation}
 where $V\in L^p$, is a
non-negative scalar function  and $\Pi$ is the
projection onto the space of functions with mean value zero:
$$
\Pi\varphi=\varphi-\frac1{4\pi}\int_{\mathbb{S}^2}\varphi(s)d\sigma.
$$
 Then ${\bf K}=
\mathbb{H}^*\mathbb{H}$ is a  compact self-adjoint operator in
 ${L}^2({\mathbb{S}}^2)$  and for $r=p'=p/(p-1)\in(1,\infty)$
$$
\aligned
\Tr \mathbf{K}^r=\Tr\left(\Pi(m^2-{\Delta})^{-1/2}V(m^2-{\Delta})^{-1/2}\Pi\right)^r\le\\\le
\Tr\left(\Pi(m^2-{\Delta})^{-r/2}V^r(m^2-{\Delta})^{-r/2}\Pi\right)=\\=
\Tr\left(V^r(m^2-{\Delta})^{-r}\Pi\right),
\endaligned
$$
where we used
 the Araki--Lieb--Thirring inequality for traces \cite{Araki, LT, traceSimon}:
$$
\Tr(BA^2B)^p\le\Tr(B^pA^{2p}B^p),\quad p\ge1,
$$
and the cyclicity property of the trace together with the facts
that $\Pi$ commutes with the Laplacian and that $\Pi$ is a
projection: $\Pi^2=\Pi$. Using the basis of orthonormal
eigenfunctions of the Laplacian \eqref{harmonics} and identity \eqref{identity},
in view of the key estimate~\eqref{2.main} proved  below we
find that
$$
\aligned
\operatorname{Tr} \mathbf{K}^r\le
\operatorname{Tr}\left(V^r(m^2-{\Delta})^{-r}\Pi\right)\\=
\int_{\mathbb{S}^2}V(s)^r\sum_{n=1}^\infty\sum_{k=1}^{2n+1}
\frac1{(m^2+n(n+1))^r}Y_n^k(s)^2d\sigma\\
=
\frac1{4\pi}\sum_{n=1}^\infty\frac{2n+1}{\bigl(m^2+n(n+1)\bigr)^r}
\int_{\mathbb{S}^2}V(s)^rd\sigma\le
\frac1{4\pi}\frac{m^{-2(r-1)}}{{r-1}}\|V\|^r_{L^r}.
\endaligned
$$
We can now argue as in~\cite{LiebJFA}. We observe that
$$
\int_{\mathbb{S}^2}\rho(s)V(s)d\sigma=\sum_{i=1}^n\|\mathbb{H}\psi_i\|^2_{L^2},
$$
where
$$
\psi_j=(m^2-{\Delta})^{1/2}\varphi_j,\quad j=1,\dots,n.
$$
Next, in view of \eqref{orth-m} the $\psi_j$'s are
orthonormal in $L^2$
$$
\aligned
(\psi_i,\psi_j)&=((m^2-\Delta)^{1/2}\varphi_i,(m^2-\Delta)^{1/2}\varphi_j)=
(\varphi_i,(m^2-\Delta)\varphi_j)\\&=m^2(\varphi_i,\varphi_j)+
(\nabla\varphi_i,\nabla\varphi_j)=\delta_{ij},
\endaligned
$$
and in view of  the variational
principle
$$
\sum_{i=1}^n\|\mathbb{H}\psi_i\|^2_{L^2}=\sum_{i=1}^n(\mathbf{K}\psi_i,\psi_i)
\le\sum_{i=1}^n\lambda_i,
$$
where $\lambda_i>0$ are the eigenvalues of the
operator $\mathbf{K}$. Therefore
$$
\aligned
\int_{\mathbb{S}^2}\rho(s)V(s)d\sigma\le\sum_{i=1}^n\lambda_i\le
n^{1/p}\left(\Tr K^r\right)^{1/r}\le\\\le  n^{1/p}
\left(\frac{p-1}{4\pi m^{2/(p-1)}}\right)^{(p-1)/p}\|V\|_{L^{p/(p-1)}}=\\=
n^{1/p}m^{-2/p}\left(\frac{p-1}{4\pi}\right)^{(p-1)/p}\|V\|_{L^{p/(p-1)}}.
\endaligned
$$
Finally, setting $V(x)=\rho(x)^{p-1}$ we obtain~\eqref{Liebd2},
\eqref{Liebbp}.
\end{proof}

\begin{corollary}\label{Cor:GN}
The following interpolation inequality holds for $\varphi\in \dot{H}^1(\mathbb {S}^2)$:
\begin{equation}\label{Gag-Nir}
\|\varphi\|_{L^q(\mathbb{S}^2)}\le\left(\frac1{4\pi}\right)^\frac{q-2}{2q}\left(\frac q2\right)^{1/2}
\|\varphi\|^{2/q}\|\nabla\varphi\|^{1-2/q},\qquad q\ge2.
\end{equation}
\end{corollary}
\begin{proof}
 For $n=1$ inequality \eqref{Liebd2} goes over to
$$
\|\varphi\|_{L^{2p}}^2\le\mathrm{B}_p
\left(m^{2-2/p}\|\varphi\|^2+m^{-2/p}\|\nabla\varphi\|^2\right).
$$
Minimizing with respect  $m$ we obtain
$$
\aligned
\|\varphi\|_{L^{2p}}^2\le\mathrm{B}_p\frac p{(p-1)^{(p-1)/p}}
\|\varphi\|^{2/p}\|\nabla\varphi\|^{2-2/p}=\\=
\left(\frac1{4\pi}\right)^{(p-1)/p}\,p\,
\|\varphi\|^{2/p}\|\nabla\varphi\|^{2-2/p},
\endaligned
$$
which is~\eqref{Gag-Nir}.
\end{proof}
The inequality for $H^1$-orthonormal divergence free vector
functions  on $\mathbb{S}^2$ and the corresponding one function
interpolation inequality are   similar to the scalar
case.

\begin{theorem}
Let a family of vector functions
$\{u_j\}_{j=1}^n\in\mathbf{H}^1_0(\Omega)$,
$\Omega\subseteq\mathbb{S}^2$ with $\div u_j=0$ be orthonormal in
$\mathbf{H}^1$:
$$
m^2(u_i,u_j)+(\rot u_i,\rot u_j)=\delta_{ij}.
$$
Then for $1\le p<\infty$
$$
\rho(x):=\sum_{j=1}^n|u_j(x)|^2
$$
satisfies
$$
\|\rho\|_{L^p}\le\mathrm{B}_pm^{-2/p}n^{1/p},
$$
where
$$
\mathrm{B}_p\le\left(\frac{p-1}{4\pi}\right)^{(p-1)/p}.
$$
\end{theorem}
\begin{proof}
The case $p=2$ was treated in~\cite{MZ}. Once we now
have~\eqref{2.main}  for  all $1<p<\infty$ the
proof of the theorem is completely analogous.
To make the paper self contained we provide some details.

In the vector case  identity~\eqref{identity} is replaced by its
vector analogue \cite{ILMS}:
\begin{equation}\label{identity-vec}
\sum_{k=1}^{2n+1}|\nabla Y_n^k(s)|^2=n(n+1)\frac{2n+1}{4\pi}.
\end{equation}
 In fact, substituting $\varphi(s)=Y_n^k(s)$ into
the identity
$$
\Delta\varphi^2=2\varphi\Delta\varphi+2|\nabla\varphi|^2
$$
we sum the results over $k=1,\dots,2n+1$. In view of \eqref{identity}
the left-hand side vanishes and we obtain~\eqref{identity-vec}
since the $Y_n^k(s)$'s are the eigenfunctions corresponding to $n(n+1)$.

Next, by the vector Laplace operator acting on (tangent)
vector fields on $\mathbb{S}^2$ we mean  the Laplace--de Rham
operator $-d\delta-\delta d$ identifying $1$-forms and vectors.
Then for a two-dimensional manifold
 we have \cite{I93}
$$
\mathbf{\Delta} u=\nabla\div u-\rot\rot u,
$$
where the  operators $\nabla=\grad$ and $\div$ have the
conventional meaning. The operator $\rot$ of a vector $u$ is a
scalar  and for a scalar $\psi$, $\rot\psi$ is a vector: $\rot
u:=\div(u^\perp)$, $\rot\psi:=\nabla^\perp\psi$, where  in the
local frame $u^\perp=(u_2,-u_1)$, that is, $\pi/2$ clockwise rotation
of $u$ in the local tangent plane. Integrating by parts we obtain
\begin{equation*}\label{byparts}
(-\mathbf{\Delta} u,u)=\|\rot u\|^2_{L^2}+\|\div u\|^2_{L^2}.
\end{equation*}

Corresponding to the eigenvalue $\Lambda_n=n(n+1)$,
where $n=1,2,\dots$, there is a family of $2n+1$ orthonormal
vector-valued  eigenfunctions  $w_n^k(s)$ of the vector Laplacian
on the invariant space of divergence free
vector-functions, that is, the Stokes operator on $\mathbb{S}^2$
\begin{equation*}\label{bases}
\aligned
w_n^k(s)&=(n(n+1))^{-1/2}\,\nabla^\perp Y_n^k(s),\ -\mathbf{\Delta}w_n^k=n(n+1)w_n^k,\
\div w_n^k=0;
\endaligned
\end{equation*}
where $k=1,\dots,2n+1$,  and~\eqref{identity-vec} implies the
following identity:
\begin{equation}\label{id-vec}
\sum_{k=1}^{2n+1}|w_n^k(s)|^2=\frac{2n+1}{4\pi}.
\end{equation}
We finally observe that  $-\mathbf{\Delta}$ is strictly positive
$-\mathbf{\Delta}\ge \Lambda_1I=2I.$

Turning to the proof we  first consider the  whole sphere
$\Omega=\mathbb{S}^2$, and as in \eqref{HH} define two operators
$$
 \mathbb{H}= V^{1/2}(m^2-\mathbf{\Delta})^{-1/2}\mathbf{\Pi},\quad
  \mathbb{H}^*=\mathbf{\Pi}(m^2-\mathbf{\Delta})^{-1/2}V^{1/2},
$$
where $\mathbf{\Pi}$ is the orthogonal Helmholtz--Leray  projection onto the subspace
$\{u\in\mathbf{L}^2(\mathbb{S}^2),\ \div u=0\}$. From this point, using
\eqref{id-vec},
we can complete the proof as in the scalar case.

Finally, if $\Omega\varsubsetneq\mathbb{S}^2$ is a proper domain on $\mathbb{S}^2$,
we extend by zero  $u_j$ outside
$\Omega$ and denote the results by $\widetilde{u}_j$,
so that $\widetilde{u}_j\in {\bf H}^1(\mathbb{S}^2)$
and $\operatorname{div}\widetilde{u}_j=0$. We
further set
$\widetilde\rho(x):=\sum_{j=1}^n|\widetilde{u}_j(x)|^2$.
Then setting $\widetilde{\psi}_i:=(m^2-\mathbf{\Delta})^{1/2}\widetilde{u}_i$,
we see that the system $\{\widetilde\psi_j\}_{j=1}^n$ is orthonormal in
$\mathbf L^2(\mathbb{S}^2)$ and
$\operatorname{div}\widetilde\psi_j=0$.
Since clearly $\|\widetilde\rho\|_{ L^2(\mathbb{S}^2)}=\|\rho\|_{ L^2(\Omega)}$,
the proof  reduces to the case of the whole sphere
and therefore is complete.
\end{proof}

\begin{remark}
{\rm
 For $q=4$ inequality \eqref{Gag-Nir} is  the Ladyzhenskaya
inequality on the 2D sphere $\mathbb{S}^2$
$$
\|\varphi\|_{L^4}\le \mathrm{c}_\mathrm{Lad}\|\varphi\|^2\|\nabla\varphi\|^2
$$
 and gives the estimate
of the constant $\mathrm{c}_\mathrm{Lad}\le1/\pi$. However,
a recent estimate of it in~\cite{ILZ-JFA}
in the terms of the  Lieb--Thirring inequality is slightly better:
  $\mathrm{c}_\mathrm{Lad}\le3\pi/32$.
On the other hand, \eqref{Gag-Nir} works for all $q\ge2$ and
provides a simple expression for the constant.
}
\end{remark}
\begin{remark}\label{R:Orlitz}
{\rm
The rate of growth as $q\to\infty$ of the constant both in \eqref{Gag-Nir} and \eqref{Gag-Nir-R2},
namely $q^{1/2}$, is optimal in the power scale. If we had not imposed the zero mean condition
for the sphere, it
would have immediately followed from~\eqref{1} with $d=2$.

In the general case, if in \eqref{Gag-Nir} and \eqref{Gag-Nir-R2} the rate of growth  was less
than $1/2$, then the Sobolev space $H^1$ in two dimensions would have been embedded
into the Orlicz space with Orlicz function
$e^{t^{2+\varepsilon}}-1$, $\varepsilon>0$, which is impossible \cite{Tr}.

Furthermore, while for every fixed $q<\infty$ the constant
 in \eqref{Gag-Nir} and \eqref{Gag-Nir-R2} is not sharp, we think,
  as mentioned before, that the \emph{sharp} constant $\mathrm{c}_q$
  behaves like
 $$
 \frac{\mathrm{c}_q}{\sqrt{q}}\to \frac1{\sqrt{8\pi}}\quad \text{as} \quad q\to\infty.
 $$

}
\end{remark}

 \setcounter{equation}{0}
\section{Proof of  estimate~\eqref{S2}}\label{Sec:3}

\begin{proposition}\label{main2}
The following inequality holds for  $p>1$ and $m\ge0$
\begin{equation}\label{2.main}
I_p(m):=m^{2(p-1)}(p-1)\sum_{n=1}^\infty\frac{2n+1}{(m^2+n^2+n)^p}<1.
\end{equation}
\end{proposition}
\begin{proof}
A  general argument  shows that inequality \eqref{S2} holds
for all sufficiently large $m$. In fact, we
 observe that we can write $I_p(m)$ in the form
$$
I_p(m)=\frac{p-1}{m^2}\sum_{n=1}^\infty (2n+1)g\left(\frac{n(n+1)}{m^2}\right),
\quad g(t)=\frac1{(1+t)^p}\,.
$$
The following asymptotic expansion as $m\to \infty$ holds for this type of series
(see~\cite[Lemma~3.5]{IZ})
$$
\aligned
I_p(m)=
(p-1)\left[\int_0^\infty g(t)dt-\frac1{m^2}\frac23 g(0)+O(m^{-4})\right]=\\=
1-\frac1{m^2}\frac{2(p-1)}3+O(m^{-4}).
\endaligned
$$
Therefore for a fixed $p>1$ there exists a sufficiently large
$m=m_p$ such that inequality \eqref{2.main} holds for
all $m\ge m_p$.

The proof that it holds for all $p>1$ and $m\ge0$ requires some specific work.
We will use the Euler--Maclaurin summation formula
(see, for example, \cite{Krylov}). Namely, we use the formula
\begin{equation}\label{Euler-MacLaurin}
\sum_{n=1}^\infty f(n)=\int_0^\infty f(x)\,dx-\frac12 f(0)-\frac1{12}f'(0)+\frac1{720}f'''(0)+R_4,
\end{equation}
with  remainder term
$$
R_4=-\frac1{4!}\int_0^\infty f^{''''}(x)B_4(x)dx,
$$
where $B_4(x)$ is the periodic Bernoulli polynomial. The remainder
term $R_4$ in this formula can be estimated as
\begin{equation}\label{R4}
|R_4|\le\frac{2\zeta(4)}{(2\pi)^4}\int_0^\infty|f''''(x)|\,dx=
\frac1{720}\int_0^\infty|f''''(x)|\,dx,
\end{equation}
where $\zeta(4)=\frac{\pi^4}{90}$ and $\zeta(s)$ is the Riemann zeta function.

We will use this formula for relatively big $m$ and
$$
f_m(x)=\frac{m^{2(p-1)}(p-1)(2x+1)}{(m^2+x^2+x)^p}.
$$
A straightforward calculation gives
$$
\aligned
&\int_0^\infty f_m(x)\,dx=1,\\
&f_m(0)=\frac{p-1}{m^2},\\
 &f_m'(0)=\frac{(p-1)(2m^2-p)}{m^4},\\
 &f_m'''(0)=-\frac{(p - 1)p(12m^4 - 12m^2p - 12m^2 + p^2 + 3p + 2)}{m^8}
\endaligned
$$
and
\begin{multline*}
f_m''''(x)=32p(p^2 - 1)m^{2p - 2}\(\frac{(x + 1/2)^5(p + 2)(p + 3)}{(m^2 + x^2 + x)^{p + 4}} -\right.\\-\left. \frac{5(x + 1/2)^3(p + 2)}{(m^2 + x^2 + x)^{p +3}} + \frac{15(x+1/2)}{4(m^2 + x^2 + x)^{p+2}}\).
\end{multline*}
We now change the sign of the second term in the above expression and set
\begin{multline*}
g(x)=32p(p^2 - 1)m^{2p - 2}\(\frac{(x + 1/2)^5(p + 2)(p + 3)}{(m^2 + x^2 + x)^{p + 4}} +\right.\\+\left. \frac{5(x + 1/2)^3(p + 2)}{(m^2 + x^2 + x)^{p +3}} + \frac{15(x+1/2)}{4(m^2 + x^2 + x)^{p+2}}\).
\end{multline*}
Then, obviously, $|f_m''''(x)|\le g(x)$ for all $x$. On the other hand, the integral of $g(x)$ can be computed explicitly (since $g(x)$ contains odd powers of $(x-1/2)$ in the numerators, hence the
corresponding antiderivatives are expressed in elementary functions):
$$
\int_0^\infty g(x)\,dx=\frac{p(p-1)\bigl(172m^4+28(p+1)m^2+p^2+3p+2\bigr)}{m^8}.
$$
Thus, the Euler--Maclaurin formula \eqref{Euler-MacLaurin} gives us
the estimate
\begin{equation}\label{E-M}
\aligned
I_p(m)<
1-\frac23(p-1)m^{-2}+\frac{11}{36}p(p-1)m^{-4}+\frac1{18}p(p^2-1)m^{-6}=\\=
1-\frac1{36}(p-1)m^{-6}\left(24m^4-11pm^2-2p(p+1)\right).
\endaligned
\end{equation}
Therefore
$$
I_p(m)<1,
$$
if $24m^4-11pm^2-2p(p+1)>0$, that is, if 
\begin{equation}\label{m0p}
m>\frac{\sqrt{3\sqrt{313p^2 + 192p} + 33p}}{12}=:m_0(p).
\end{equation}

We now consider two cases: $p\in(1,2]$ and $p>2$. So let $p\in(1,2]$.
The maximum value of $m_0(p)$ on  $p\in(1,2]$  is attained at $p=2$,
so we have proved the desired inequality~\eqref{2.main} for all $p\in (1,2]$ and
\begin{equation*}\label{m0}
m>m_0:=\frac{\sqrt{66 + 6\sqrt{409}}}{12}\approx1.1406.
\end{equation*}
Thus, we only need to verify the desired inequality for $m<m_0$.
We single out the first term in the series and
drop the the dependence on $m$ in the remaining terms. We obtain
\begin{equation}\label{Gmp0}
\aligned
I_p(m)=m^{2(p-1)}(p-1)\(\frac3{(m^2+2)^p}+\sum_{n=2}^\infty\frac{2n+1}{(m^2+n^2+n)^p}\)<\\<
m^{2(p-1)}(p-1)\(\frac3{(m^2+2)^p}+\sum_{n=2}^\infty\frac{2n+1}{(n^2+n)^p}\)=\\=
m^{2(p-1)}(p-1)\(\frac3{(m^2+2)^p}+ R(p)\)=:G(m,p),
\endaligned
\end{equation}
where
$$
R(p):=\sum_{n=2}^\infty\frac{2n+1}{(n^2+n)^p}\,.
$$

 To complete the proof, we only need to prove  the inequality
$$
G(m,p)<1
$$
for all $p\in[1,2]$ and all $m\in[0,m_0]$.

We again apply the Euler--Maclaurin formula to the series $R(p)$
(taking into account that the summation now starts with $n=2$).
Setting
$$
f(n):=\frac{2n+1}{(n^2+n)^p}
$$
we have
$$
\aligned
\int_1^\infty f(x)\,dx=
\frac{2^{1-p}}{p-1},\ \
f(1)=\frac3{2^p}, \ \
f'(1)=\frac1{2^p}\left(2-\frac{9p}2\right),
\endaligned
$$
and
$$
\aligned
f''''(n)=
\frac{(2n+1)p(p+1)}{(n^2+n)^{p+4}}
\biggl(
(16p^2-4)n^4+(32p^2-8)n^3+\\+(24p^2+20p+4)n^2
+(8p^2+20p+8)n+p^2+5p+6\biggr).
\endaligned
$$
Since clearly $f''''(n)>0$, it follows from~\eqref{R4} that the
last two terms in the Euler--Maclaurin formula add up to zero:
$$
\frac1{720}f'''(1)+R_4\le\frac1{720}\left(f'''(1)+\int_1^\infty f''''(x)\,dx\right)=0.
$$
Therefore
$$
R(p)\le\int_1^\infty f(x)-\frac12f(1)-\frac1{12}f'(1)=
\frac1{2^p(p-1)}\left(\frac{9p^2-49p+88}{24}\right).
$$
We substitute this  into the expression for $G(m,p)$ and
set $z:=m^2/2$. We further suppose that
$z\le1$.  Then, since $e^{-x}\le1/(1+x)$, $x\ge0$
and taking into account that $\ln z\le0$
we have
$$
z^{p-1}=e^{(p-1)\ln z}<\frac1{1-(p-1)\ln z}.
$$
Using this  and   the Bernoulli inequality $(1+z)^p>1+pz$, we
obtain
\begin{equation}\label{Gmp1}
\aligned
G(m,p)-1=\frac12z^{p-1}\left(\frac{3(p-1)}{(1+z)^p}+2^p(p-1)R(p)\right)-1<\\
\frac12\frac1{(1-(p-1)\ln z)}\left(\frac{3(p-1)}{1+pz}+\frac{9p^2-49p+88}{24}\right)-1=\\
\frac{p-1}{48(pz+1)(1-(p-1)\ln z)}\times\\
\biggl(9zp^2+(48 z\ln z-40 z+9)p
+48 \ln z+32\biggr)=:
A(z,p)\phi(z,p).
\endaligned
\end{equation}
For the  future reference we point out that inequality~\eqref{Gmp1}
holds for all $m\le\sqrt{2}$ (so that $z\le1$) and {\it all} $p>1$
and in this case $A(z,p)>0$. Therefore the sign of $G(m,p)-1$
coincides with that of the quadratic polynomial $\phi(z,p)$.
Furthermore, for a fixed $p$ the function $\phi(z,p)$ is  monotone
increasing with respect to $z$. In fact, since $p\ln z+1/z\ge
p(1-\ln p)$, we have for $p>1$
\begin{equation}\label{monotfi}
\partial_z\phi(z,p)=9p^2+8p+48\left(p\ln z+\frac1z\right)\ge
9p^2+56p-48p\ln p>0.
\end{equation}

Returning now to the case $p\in (1,2]$ we observe that
 for $z=m_0^2/2=0.6504<1$, $\ln z<0$ we have
$$
\phi(m_0^2/2,p)=5.854\,p^2-30.446\,p+11.358<0\quad \text{for}\quad p\in[1,2].
$$
Hence
$$
\phi(m^2/2,p)<0\quad\text{for all}\quad m\in[0,m_0] \quad\text{and}\quad p\in[1,2].
$$
This completes the proof of inequality~\eqref{2.main} for $p\in
(1,2]$.

We are now ready to verify inequality~\eqref{2.main} for $p>2$ as well.
The key idea here is to use the fact that $I_p(m)$ is monotone decreasing
with respect to $p$ for $0\le m\le m_1(p)$, where $m_1(p)$ is given below.
Indeed, let
$$
f(n)=f(m,n,p)=\frac{m^{2(p-1)}(p-1)(2n+1)}{(m^2+n^2+n)^p}.
$$
Then
$$
\partial_pf(n)=\frac{m^{2(p-1)}(2n+1)}{(m^2+n^2+n)^p}
\left(1+(p-1)\ln\frac{m^2}{m^2+n^2+n}\right),
$$
and we see that the derivative is negative for all $n\in\mathbb N$
if
$$
m<m_1(p):=\frac{\sqrt{2}}{\sqrt{e^{\frac1{p-1}}-1}}\,.
$$

Let now $p>2$ be fixed. Two cases are possible
$$
m_0(p)<m_1(p)\quad\text{and}\quad m_1(p)\le m_0(p).
$$
In the first case inequality \eqref{2.main} holds for all $m$, since
if $m>m_0(p)$, it holds in view of~\eqref{E-M}, \eqref{m0p},
while if $m<m_0(p)<m_1(p)$, it holds in view of the established monotonicity
with respect to $p$
and the fact that \eqref{2.main} holds for $p=2$.
\begin{figure}[htb]
\centerline{\psfig{file=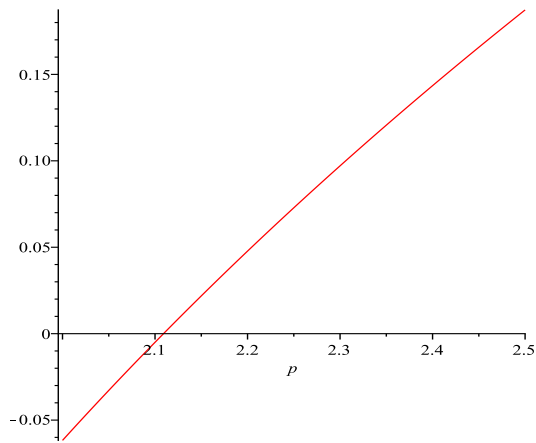,width=7.5cm,height=6cm,angle=0}}
\caption{The graph of $m_1(p)-m_0(p)$ on $p\in[2,2.5]$.}
\label{fig:T2}
\end{figure}

In the second case we first find the interval with respect to $p$ where
the inequality $m_1(p)\le m_0(p)$ actually holds. Namely, it holds for
$$
p\in[2,p_*],\quad p_*=2.10915\dots\,,
$$
see Fig.~\ref{fig:T2}, where the unique $p_*$ is found numerically.

Thus,  inequality~\eqref{2.main} holds  for $p>p_*$ and
we only need to look at the interval $p\in[2,p_*]$.
Furthermore, since $m_0(p)$ in \eqref{m0p} is monotone increasing,
we only need to check~\eqref{2.main} for
$$
p\in[2,p_*]\quad\text{and}\quad m\in[0,m_*],\ m_*=m_0(p_*)=1.169\,\dots.
$$
In view of \eqref{Gmp0}, \eqref{Gmp1}, \eqref{monotfi} and the
remark after \eqref{Gmp1} we have the following  sequence of
implications
$$
\aligned
\left\{I_p(m)<1\right\}\Leftarrow\left\{G(m,p)-1<0\right\}\Leftarrow
\left\{A(z,p)\phi(z,p)<0\right\}\Leftrightarrow\\
\left\{\phi(z,p)<0\right\}\Leftarrow\left\{\phi(z_*,p)<0\right\}\Leftrightarrow\\
\left\{6.1495p^2-30.8222p+13.7197<0,\ p\in[2,p_*]\right\}=\{'\text{true}'\},
\endaligned
$$
where $z=m^2/2\le z_*=m_*^2/2=0.6832<1$, and $m_*=1.169<\sqrt{2}$.
Inequality \eqref{2.main} is now proved for the whole range of parameters and
the proof is complete.
\end{proof}

\begin{remark}\label{R:MZ}
{\rm
The case $p=2$ important for applications was treated by more elementary means in
\cite{MZ}.
}
\end{remark}

\begin{remark}\label{R:monotone}
{\rm
Calculations show that for each $p$ tested, the function $I_p(m)$ is monotone
increasing with respect to $m$. We are not able to prove it
rigorously
at the moment. However, it was shown in \cite{Lieb90arxiv} that
the lattice sum $J_p(m)$ in \eqref{T2} is monotone increasing  in $m$,
which obviously implies inequality \eqref{T2}, since
$J_p(\infty)=1$.
}
\end{remark}

\end{document}